\documentclass[11pt]{article}

\usepackage{amssymb}
\usepackage{amsthm}
\usepackage{amsmath}

\newtheorem{theorem}{Theorem}[section]
\newtheorem{lemma}{Lemma}[section]
\newtheorem{proposition}{Proposition}[section]
\newtheorem{corollary}{Corollary}[section]

\numberwithin{equation}{section}

\newcommand{\bmu}{\boldsymbol{\mu}}
\newcommand{\bgam}{\boldsymbol{\gamma}}

\newcommand{\FF}{\mathbb{F}}  
\newcommand{\GL}{\mathrm{GL}}
\newcommand{\U}{\mathrm{U}}

\newcommand{\Sp}{\mathrm{Sp}}
\newcommand{\gO}{\mathrm{O}}

\newcommand{\cU}{\mathcal{U}}
\newcommand{\cP}{\mathcal{P}}
\newcommand{\tf}{\tilde{f}}
\newcommand{\tr}{\mathrm{tr}}

\def\adots{\mathinner{\mkern2mu\raise0pt\hbox{.}  
\mkern2mu\raise4pt\hbox{.}\mkern1mu
\raise7pt\vbox{\kern7pt\hbox{.}}\mkern1mu}}

\allowdisplaybreaks[1]

\begin{document}

\bibliographystyle{amsplain}

\title{Strongly real classes in finite unitary groups of odd characteristic}
\author{Zachary Gates, Anupam Singh, and C. Ryan Vinroot}
\date{}

\maketitle

\begin{abstract}
We classify all strongly real conjugacy classes of the finite unitary group $\U(n, \FF_q)$ when $q$ is odd.  In particular, we show that $g \in \U(n, \FF_q)$ is strongly real if and only if $g$ is an element of some embedded orthogonal group $\gO^{\pm}(n, \FF_q)$.  Equivalently, $g$ is strongly real in $\U(n, \FF_q)$ if and only if $g$ is real and every elementary divisor of $g$ of the form $(t \pm 1)^{2m}$ has even multiplicity.  We apply this to obtain partial results on strongly real classes in the finite symplectic group $\Sp(2n, \FF_q)$, $q$ odd, and a generating function for the number of strongly real classes in $\U(n, \FF_q)$, $q$ odd, and we also give partial results on strongly real classes in $\U(n, \FF_q)$ when $q$ is even.
\end{abstract}

\section{Introduction} \label{Intro}

An element $g$ of a group $G$ is called \emph{real} in $G$ if $g$ is conjugate to $g^{-1}$ in $G$.  A real element $g$ of $G$ is called \emph{strongly real} in $G$ if there exists some $s \in G$ such that $s^2 = 1$ and $sgs=g^{-1}$.  Equivalently, $g \in G$ is strongly real if we may write $g = s_1 s_2$ for some $s_1, s_2 \in G$ such that $s_1^2 = s_2^2 = 1$.  Since these properties are invariant under conjugation in $G$, we may speak of real and strongly real conjugacy classes of $G$.  The terminology comes from the representation theory of finite groups, as it is known that the number of real classes of a finite group $G$ is equal to the number of irreducible complex characters of $G$ which are real-valued.  There are connections, found by Gow \cite{Go79}, between the strongly real classes of $G$ and the real-valued irreducible characters of $G$ which are characters of real representations, although it is known that these two sets are not in bijection in general.  It is one long-term goal to better understand the connection between strongly real classes of a finite group, and the irreducible complex representations of that group which are realizable over the real numbers.

There is an active program of classifying the real and strongly real classes of finite groups.  Tiep and Zalesski \cite{TiZa05} have classified all finite simple and quasi-simple groups with the property that all of their classes are real.  Vdovin and Gal't \cite{VdGa10} have finished the proof that for any finite simple group $G$ with the property that all of its classes are real, must also have the property that all of its classes are strongly real, which requires a case-by-case analysis, including results in papers of Ellers and Nolte \cite{ElNo82}, Kolesnikov and Nuzhin \cite{KoNu05}, Gazdanova and Nuzhin \cite{GaNu06}, Gal't \cite{Ga10}, and R\"am\"o \cite{Ra11}.  Since these cases naturally center around the study of finite simple groups of Lie type, there is interest in understanding the real and strongly real classes of finite groups of Lie type in general.  Gill and the second-named author of this paper have completely classified the real and strongly real classes of the finite special linear groups, the finite projective linear groups, and the quasi-simple covers of the finite projective special linear groups \cite{GiSi11I, GiSi11II}.

The main result of this paper, given in Theorem \ref{MainThm}, is a classification of the strongly real classes of the (full) unitary group $\U(n, \FF_q)$ over a finite field $\FF_q$ with $q$ elements, where $q$ is odd.  The case of the finite unitary group is of particular interest, since it is closely related to the finite general linear group $\GL(n, \FF_q)$ in structure, and in fact the real classes of $\GL(n, \FF_q)$ are in natural bijection with the real classes of $\U(n, \FF_q)$ (noticed by Gow \cite{Go84}).  However, while it is known that all real classes of $\GL(n, \FF_q)$ are strongly real, this does not hold for the finite unitary group.  The problem of finding real classes of $\U(n, \FF_q)$ which are not strongly real was addressed in part in a paper of Gow and the third-named author of this paper \cite{GoVi08}, with applications to character theory, where the classification question is answered for regular unipotent elements in the finite unitary group.

The organization of this paper is as follows.  In Section \ref{ConjClasses}, we first establish notation, give a parametrization of the conjugacy classes of $\U(n, \FF_q)$, and explain which of these conjugacy classes are real.  We then give two known results, Propositions \ref{oneway} and \ref{onewayeven}, which give subsets of these real classes which are known to be strongly real, in the cases that $q$ are odd and even, respectively.  Proposition \ref{oneway} states that when $q$ is odd, if an element $g \in \U(n, \FF_q)$ is conjugate to an element of an embedded orthogonal group $\gO^{\pm}(n, \FF_q)$, then $g$ is strongly real.  Our main result is precisely the converse of this statement.  Proposition \ref{onewayeven} states that when $q$ is even, if an element $g \in \U(2n, \FF_q)$ is conjugate to an element of an embedded symplectic group $\Sp(2n, \FF_q)$, then $g$ is strongly real.  The converse of this statement is false, however, which we show in Proposition \ref{31}.

In Section \ref{Unipotent}, we reduce the problem of classifying strongly real classes in $\U(n, \FF_q)$ to classifying the unipotent strongly real classes, for any $q$, in Propositions \ref{UniRed} and \ref{UniPlusMinus}.  The main work is then in Section \ref{Induction}, where we concentrate on unipotent classes.  The main tool for our argument is Proposition \ref{IndStep}, which roughly states that if a unipotent class of one type is strongly real in $\U(n, \FF_q)$, then specific unipotent classes in $\U(n^{\#}, \FF_q)$, for certain $n^{\#} < n$, are also strongly real.  This statement allows for an induction argument, in that if we know certain small unipotent classes are not strongly real, then we may conclude that many larger ones are also not strongly real.  This argument is a generalization of the ideas used in the induction proof of \cite[Proposition 5.1]{GoVi08}, where it is shown that regular unipotent elements in $\U(n, \FF_q)$ are not strongly real if either $n$ is even and $q$ is odd, or $n$ is odd and $q$ is even.

In Section \ref{Main}, we prove our main results.  As mentioned above, the proof of the main theorem, Theorem \ref{MainThm}, is an induction proof using Proposition \ref{IndStep}.  We are reduced to proving that unipotent classes of type $(2^{m_2} 1^{m_1})$ are not strongly real in Lemma \ref{TwoLemma}, which is where we must use the assumption that $q$ is odd.  After the proof of the main result, we are able to obtain Corollary \ref{SpCor}, which gives certain classes of the symplectic group $\Sp(2n, \FF_q)$, $q$ odd, which are not strongly real.  We also give a generating function for the number of the strongly real classes in $\U(n, \FF_q)$, $q$ odd, in Corollary \ref{enumeration}.  Finally, in Section \ref{CharTwo}, we conclude with some partial results on strongly real classes in $\U(n, \FF_q)$ in the case that $q$ is even.\\
\\
\noindent{\bf Acknowledgments. } The third-named author was supported by NSF grant DMS-0854849.

\section{Real conjugacy classes of finite unitary groups} \label{ConjClasses}

\subsection{Finite unitary groups} \label{UnDef}

Let $\FF_q$ be a finite field with $q$ elements, where $q$ is the power of a prime $p$.  We will not restrict $p$ at the moment, as many results we state hold for all $p$.  Let $\FF_{q^2}$ be the quadratic extension of $\FF_q$, with $F: a \mapsto a^q$ the unique nontrivial automorphism of $\FF_{q^2}/\FF_q$, which we will also denote by $F(a) = \bar{a}$.  Given some matrix $g = (a_{ij})$ with $a_{ij} \in \FF_{q^2}$, we also write $\bar{g} = (a_{ij}^q)$.  We define the finite unitary group defined over $\FF_q$, denoted $\U(n, \FF_q)$, as
$$ \U(n, \FF_q) = \{ g \in \GL(n, \FF_{q^2}) \, \mid \, {^\top g}^{-1} = \bar{g} \}.$$
If $V$ is an $n$-dimensional $\FF_{q^2}$-vector space, define $H: V \times V \rightarrow \FF_{q^2}$ by $H(v,w) = {^\top \bar{v}}w$, so that $H$ is a non-degenerate Hermitian form on $V$, that is, $H$ is $\FF_{q^2}/\FF_q$-sesquilinear and satisfies $H(v,w) = \overline{H(w,v)}$ for all $u, v, w \in V$.  We may equivalently define $\U(n, \FF_q)$ as 
$$ \U(n, \FF_q) = \{ g \in \GL(n, \FF_{q^2}) \, \mid \, H(gv,gw) = H(v,w) \text{ for all } v, w \in V \}.$$
In fact, we may replace $H$ by any non-degenerate Hermitian form on $V$, and the stabilizing group of such a form in $\GL(n, \FF_{q^2})$ is still isomorphic to $\U(n, \FF_q)$ as defined above \cite[Corollary 10.4]{Gr02}.  In terms of matrices, this means that if $J$ is any invertible Hermitian $n$-by-$n$ matrix over $\FF_{q^2}$, so satisfies ${^\top J} = \bar{J}$, then we also have
\begin{equation} \label{Umatrix}
\U(n, \FF_q) \cong \{ g \in \GL(n, \FF_{q^2}) \, \mid \, {^\top \bar{g}} J g = J \}.
\end{equation}

In the case that $q$ is odd, any non-degenerate symmetric bilinear form on an $n$-dimensional $\FF_q$-vector space  may be extended to a non-degenerate Hermitian form on an $n$-dimensional $\FF_{q^2}$-vector space.  Also, any isometry of the symmetric form may be extended to an isometry of the Hermitian form, which means we may embed any finite orthogonal group $\gO^{\pm}(n, \FF_q)$ (either split or non-split) in the group $\U(n,\FF_q)$.

Similarly, in the case that $q$ is even, any non-degenerate symplectic bilinear form on a $2n$-dimensional $\FF_q$-vector space may be extended to a non-degenerate Hermitian form on a $2n$-dimensional $\FF_{q^2}$-vector space, and any isometry of that symplectic form may be extended to an isometry of the Hermitian form.  Thus, we may embed the finite symplectic group $\Sp(2n, \FF_q)$, $q$ even, in the group $\U(2n, \FF_q)$.

\subsection{Conjugacy classes in finite unitary groups} \label{Classes}

The conjugacy classes of $\U(n, \FF_q)$, as described in \cite{En62, Wa62}, are determined by the theory of elementary divisors, similar to $\GL(n, \FF_q)$.  Let $\bar{\FF}_q$ be a fixed algebraic closure of $\FF_q$, and extend the definition of $F$ on $\bar{\FF}_q$ as $F(a) = a^q$.  Define the map $F_U$ on $\bar{\FF}_q$ as $F_U(a) = a^{-q}$.  Recall that monic irreducible polynomials with nonzero constant in $\FF_q[t]$ are in bijection with $F$-orbits of $\bar{\FF}_q^{\times}$, via the correspondence
$$ \{ a, a^q, a^{q^2}, \ldots, a^{q^{d-1}} \} \longleftrightarrow (t-a)(t-a^q) \cdots (t - a^{q^{d-1}}),$$
and lists of powers of irreducible polynomials, or elementary divisors, whose product has degree $n$, determine conjugacy classes of $\GL(n, \FF_q)$.  For the group $\U(n, \FF_q)$, we obtain conjugacy classes by replacing $F$ by $F_U$.  That is, consider $F_U$-orbits of $\bar{\FF}_q^{\times}$, and the correspondence
\begin{equation} \label{Uirred}
\{a, a^{-q}, a^{q^2}, \ldots, a^{(-q)^{d-1}} \} \longleftrightarrow (t-a)(t-a^{-q})(t-a^{q^2})\cdots (t - a^{(-q)^{d-1}}).
\end{equation}
The polynomials obtained in this way are monic polynomials in $\FF_{q^2}[t]$ with nonzero constant which we call \emph{$U$-irreducible polynomials} (following \cite{En62}).  The conjugacy classes of $\U(n, \FF_q)$ then correspond to lists of powers of $U$-irreducible polynomials (the elementary divisors in this case), the product of which has degree $n$.

To be more precise, let $\cU$ be the set of $U$-irreducible polynomials in $\FF_{q^2}[t]$, and let $\cP$ be the set of all partitions of non-negative integers.  Given $f \in \cU$, let $d(f)$ denote its degree, and given a partition $\mu \in \cP$, with $\mu = (\mu_1, \mu_2, \ldots, \mu_l)$, let $|\mu| = \sum_{i=1}^l \mu_i$.  Let $\emptyset \in \cP$ denote the empty partition, where $|\emptyset| = 0$.  The conjugacy classes of $\U(n, \FF_q)$ are parametrized by functions $\bmu: \cU \rightarrow \cP$ such that
$$ \sum_{f \in \cU} d(f) |\bmu(f)| = n.$$
Let $c_{\bmu}$ denote the conjugacy class parametrized by $\bmu$, and let $\cU_{\bmu}$ be the support of $\bmu$, the set of all $f \in \cU$ such that $\bmu(f) \neq \emptyset$.  Given any $f \in \cU_{\bmu}$, with $\bmu(f) = (\bmu(f)_1, \bmu(f)_2, \ldots, \bmu(f)_l)$, any element $g \in c_{\bmu}$ has as elementary divisors
$$ f^{\bmu(f)_1}, f^{\bmu(f)_2}, \ldots, f^{\bmu(f)_l},$$
and moreover,
$$ \prod_{f \in \cU_{\bmu}} \prod_i f^{\bmu(f)_i}$$
is the characteristic polynomial of $g \in \U(n, \FF_q)$.

Given any conjugacy class $c_{\bmu}$ of $\U(n, \FF_q)$, let $n_f = d(f) |\bmu(f)|$, and let $\bmu_f: \cU \rightarrow \cP$ be defined by $\bmu_f(f) = \bmu(f)$, and $\bmu_f(h) = \emptyset$ for any $h \in \cU$, $h \neq f$.   Then $c_{\bmu}$ contains an element $g$ which is in block diagonal form, $g = (g_f)_{f \in \cU_{\bmu}}$, where $g_f \in \U(n_f, \FF_q)$ such that $g_f \in c_{\bmu_f}$.  In the case $f = t-1$, $g_f$ is a unipotent element of $\U(n_f, \FF_q)$.  Given this notation, let $C(g)$ denote the centralizer of $g$ in $\U(n, \FF_q)$, and let $C_f(g_f)$ denote the centralizer of $g_f$ in $\U(n_f, \FF_q)$.  The following comes directly from the description of the orders of centralizers in finite unitary groups due to Wall \cite[Sec. 2.6, Case (A), proof of (iv)]{Wa62}.

\begin{proposition} \label{Cent}
Let $g \in \U(n, \FF_q)$ with $g \in c_{\bmu}$.  The centralizer of $g = (g_f)_{f \in \cU}$ in $\U(n, \FF_q)$ is the direct product of the centralizers of $g_f$ in $\U(n_f, \FF_q)$.  That is,
$$ C(g) \cong \prod_{f \in \cU_{\bmu}} C_f(g_f).$$
\end{proposition}

\subsection{Real conjugacy classes} \label{RealClasses}

The real conjugacy classes of $\U(n, \FF_q)$, given in \cite{Go84}, are described as follows.  Let $f \in \cU$ correspond to the $F_U$-orbit of $a \in \bar{\FF}_q^{\times}$ as in (\ref{Uirred}), so $f(t) = (t-a)\cdots(t-a^{(-q)^{d-1}})$.  Define $\tf \in \cU$ to be the $U$-irreducible polynomial corresponding to the $F_U$-orbit of $a^{-1}$, so $\tf(t) = (t-a^{-1}) (t-a^q) \cdots (t-a^{-(-q)^{d-1}})$.  Then the real conjugacy classes of $\U(n, \FF_q)$ are exactly those $c_{\bmu}$ such that $\bmu(f) = \bmu(\tf)$ for all $f \in \cU$.  Note that we may have $f = \tf$, such as the case $f = t \pm 1$.  If $u(t)$ is any product of $U$-irreducible polynomials, then we may define $\tilde{u}$ to be the product of $\tf$ for all $U$-irreducible factors of $u(t)$ (the factorization of which is unique \cite{En62}).

If $h(t)$ is a monic irreducible polynomial in $\FF_q[t]$ corresponding to the $F$-orbit of $a \in \bar{\FF}_q^{\times}$, we may similarly define $\tilde{h}(t)$ as the monic irreducible polynomial with nonzero constant corresponding to the $F$-orbit of $a^{-1}$.  We may then extend the definition of $\tilde{H}(t)$ for any monic $H(t) \in \FF_q[t]$ with nonzero constant as the product of $\tilde{h}(t)$ for all irreducible factors $h(t)$ of $H(t)$, and moreover this is consistent with the extension to products of $U$-irreducible polynomials given in the previous paragraph.  It is observed in \cite{Go84} that the real conjugacy classes of $\U(n, \FF_q)$ are in bijection with real conjugacy classes of $\GL(n, \FF_q)$, which follows from the fact that for any $a \in \bar{\FF}_q^{\times}$, the union of the $F_U$-orbits of $a$ and $a^{-1}$ is equal to the union of the $F$-orbits of $a$ and $a^{-1}$.  In particular, any polynomial $u(t) \in \FF_{q^2}[t]$ which is the product of $U$-irreducible polynomials and satisfies $\tilde{u} = u$, is a monic polynomial with nonzero constant in $\FF_q[t]$ which satisfies $\tilde{u} = u$, and vice versa.

As in Section \ref{UnDef}, in the case that $q$ is odd, we may embed any finite orthogonal group $\gO^{\pm}(n, \FF_q)$ in $\U(n, \FF_q)$.  The following results are due to Wonenburger \cite{Wo66} and Wall \cite[Sec. 2.6, Cases (A) and (C)]{Wa62}, respectively.

\begin{proposition} \label{orthog} Let $q$ be the power of an odd prime.
\begin{enumerate}
\item[(i)]  Every element in an orthogonal group $\gO^{\pm}(n, \FF_q)$ is strongly real.
\item[(ii)]  An element $g$ of $\U(n, \FF_q)$ (or $\GL(n, \FF_q)$) is an element of an embedded orthogonal group $\gO^{\pm}(n, \FF_q)$ if and only if $g$ is real and every elementary divisor of $g$ of the form $(t \pm 1)^{2m}$ occurs with even multiplicity.
\end{enumerate}
\end{proposition}

\noindent
Another way of stating Proposition \ref{orthog}(ii) is that a conjugacy class $c_{\bmu}$ in $\U(n, \FF_q)$ consists of elements from a subgroup isomorphic to an orthogonal group $\gO^{\pm}(n, \FF_q)$ if and only if $\bmu(f) = \bmu(\tf)$ for every $f \in \cU$, and for $f = t \pm 1$, any even part of $\bmu(f)$ has even multiplicity.

The next statement, which follows directly from Proposition \ref{orthog}, is given in \cite[Proposition 5.2(a)]{GoVi08}.

\begin{proposition} \label{oneway} Let $q$ be the power of an odd prime.  Let $g \in \U(n, \FF_q)$ be a real element such that every elementary divisor of $g$ of the form $(t \pm 1)^{2m}$ occurs with even multiplicity.  Then $g$ is strongly real.
\end{proposition}

\noindent 
The main result of this paper, then, is the converse of Proposition \ref{oneway}.

In the case that $q$ is even, we may embed $\Sp(2n, \FF_q)$ in $\U(2n, \FF_q)$.  The following results, where (i) was obtained by Ellers and Nolte \cite{ElNo82} and Gow \cite{Go81}, and (ii) follows from \cite{He79} or \cite[Sec. 3.7]{Wa62}, are results which parallel those in Proposition \ref{orthog}.

\begin{proposition} \label{symplec} Let $q$ be a power of $2$.
\begin{enumerate}
\item[(i)]  Every element in the group $\Sp(2n, \FF_q)$ is strongly real.
\item[(ii)]  An element $g$ of $\U(2n, \FF_q)$ (or $\GL(2n, \FF_q)$) is an element of a embedded $\Sp(2n, \FF_q)$ if and only if $g$ is real and every elementary divisor of $g$ of the form $(t-1)^{2m+1}$, $m \geq 1$, occurs with even multiplicity.
\end{enumerate}
\end{proposition}

That is, by Proposition \ref{symplec}(ii), a conjugacy class $c_{\bmu}$ in $\U(2n, \FF_q)$, $q$ even, consists of elements from a subgroup isomorphic to $\Sp(2n, \FF_q)$ if and only if $\bmu(f) = \bmu(\tf)$ for every $f \in \cU$, and any odd part of $\bmu(t-1)$ greater than 1 has even multiplicity.  So, from Proposition \ref{symplec}, we have the following statement, which is \cite[Proposition 5.2(b)]{GoVi08}.

\begin{proposition} \label{onewayeven} Let $q$ be a power of $2$.  Let $g \in \U(2n, \FF_q)$ be a real element such that every elementary divisor of $g$ of the form $(t-1)^{2m+1}$, $m \geq 1$, occurs with even multiplicity.  Then $g$ is strongly real.
\end{proposition}

Unlike Proposition \ref{oneway}, the converse of Proposition \ref{onewayeven} is false, which we show in Section \ref{CharTwo}.

\section{Reduction to unipotent elements} \label{Unipotent}

In this section, we reduce the proof of the main theorem to proving the statement only for unipotent elements.  We may prove this reduction statement without restricting $q$, as follows.  Recall the following notation from Section \ref{Classes}.  Given $f \in \cU$, and $\bmu: \cU \rightarrow \cP$, $\cU_{\bmu}$ is the set of all $f \in \cU$ such that $\bmu(f) \neq \emptyset$, and for any $f \in \cU$, $\bmu_f : \cU \rightarrow \cP$ is defined by $\bmu_f(f) = \bmu(f)$, and $\bmu_f(h) = \emptyset$ for any $h \neq f$.

\begin{proposition} \label{UniRed} Fix $q$ the power of any prime.  Let $c_{\bmu}$ be a real conjugacy class in $\U(n, \FF_q)$.  The conjugacy class $c_{\bmu}$ is strongly real in $\U(n, \FF_q)$ if and only if the conjugacy classes $c_{\bmu_{t-1}}$ and $c_{\bmu_{t+1}}$ are strongly real in $\U(n_{t-1}, \FF_q)$ and $\U(n_{t+1}, \FF_q)$, respectively.
\end{proposition}
\begin{proof} Let $c_{\bmu}$ be a real conjugacy class of $\U(n, \FF_q)$, and let $g = (g_f)_{f \in \cU_{\bmu}}$ be a block diagonal element of $c_{\bmu}$, where $g_f \in \U(n_f, \FF_q)$, with $g_f \in c_{\bmu_f}$, and $n_f = d(f) |\bmu(f)|$, as in Section \ref{Classes}.  We must show that $g$ is strongly real if and only if $g_{t-1}$ and $g_{t+1}$ are strongly real.

Given the real conjugacy class $c_{\bmu}$, and the block diagonal element $g = (g_f) \in c_{\bmu}$, for any $f \in \cU_{\bmu}$, define $g_{f^*} = g_f$ whenever $f = \tf$, and in the case $f \neq \tf$, define $g_{f^*}$ as the block diagonal element $\left( \begin{array}{cc} g_f &  \\   & g_{\tf} \end{array} \right)$.  Likewise, define $n_{f^*} = n_f$ if $f = \tf$, and define $n_{f^*} = n_f + n_{\tf}$ if $f \neq \tf$.  With this notation, we may write $g$ as a block diagonal element $g = (g_f)_{f \in \cU} = (g_{f^*})$, where $g_{f^*}$ is a real element of $\U(n_{f^*}, \FF_q)$.  Note that, whenever $f \neq t \pm 1$, but $f = \tf$, then the degree of $f$ is even.  Thus, if $f \neq t \pm 1$, then $f^*$ has even degree, and it follows from Proposition \ref{oneway} (if $q$ is odd) and Proposition \ref{onewayeven} (if $q$ is even) that $g_{f^*}$ is strongly real in $\U(n_{f^*}, \FF_q)$.  For each such $f^*$, let $s_{f^*} \in \U(n_{f^*}, \FF_q)$ be the element with the property that $s_{f^*}^2 = 1$ and $s_{f^*} g_{f^*} s_{f^*} = g_{f^*}^{-1}$.

Suppose that $g_{t-1}$ and $g_{t+1}$ are strongly real, and let $s_{t \pm 1} \in \U(n_{t \pm 1}, \FF_q)$ be the elements such that $s_{t \pm 1}^2 = 1$ and $s_{t \pm 1} g_{t \pm 1} s_{t \pm 1} = g_{t \pm 1}^{-1}$.  Now define the block diagonal element $s = (s_{f^*}) \in \prod_{f^*} \U(n_{f^*}, \FF_q) \subset \U(n, \FF_q)$.  Then $s$ satisfies $s^2 = 1$ and $s g s = g^{-1}$, so $g$ is strongly real in $\U(n, \FF_q)$.

Conversely, suppose that $g = (g_{f^*}) \in \prod_{f^*} \U(n_{f^*}, \FF_q)$ is strongly real in $\U(n, \FF_q)$, and let $s \in \U(n, \FF_q)$ be such that $s^2 = 1$ and $s g s = g^{-1}$.  Since each $g_{f^*}$ is a real element of $\U(n_{f^*}, \FF_q)$, then for each $f^*$, there is some $b_{f^*} \in \U(n_{f^*}, \FF_q)$ such that $b_{f^*} g_{f^*} b_{f^*}^{-1} = g_{f^*}^{-1}$.  Letting $b$ be the block diagonal element $b = (b_{f^*})$, it follows that we must have $s \in b C(g)$, where $C(g)$ is the centralizer of $g$ in $\U(n, \FF_q)$.  It follows from Proposition \ref{Cent} that $C(g) = \prod_{f^*} C_{f^*}(g_{f^*})$, where $C_{f^*}(g_{f^*})$ is the centralizer of $g_{f^*}$ in $\U(n_{f^*}, \FF_q)$.  Thus, we have
$$ s \in b C(g) = \prod_{f^*} b_{f^*} C_{f^*}(g_{f^*}).$$
That is, $s$ must be in block diagonal form itself, so $s = (s_{f^*}) \in \prod_{f^*} \U(n_{f^*}, \FF_q)$.  For each $f^*$, we must then have $s_{f^*}^2 = 1$ and $s_{f^*} g_{f^*} s_{f^*} = g_{f^*}^{-1}$.  In particular, taking $f = t \pm 1$, we may conclude that the elements $g_{t \pm 1}$ are strongly real in $\U(n_{t \pm 1}, \FF_q)$.
\end{proof}

In the case that $q$ is even, since $t+1 = t-1$, Proposition \ref{UniRed} reduces the classification of strongly real classes in $\U(n, \FF_q)$ to the classification of strongly real unipotent classes in $\U(n, \FF_q)$.  When $q$ is odd, we say a class $c_{\bgam}$ of $\U(n,\FF_q)$ is \emph{negative unipotent} if $\bgam(f) = \emptyset$ whenever $f \neq t+1$.  To reduce to the unipotent case when $q$ is odd, we need the following.

\begin{proposition} \label{UniPlusMinus} Let $\cP^+$ be the set of partitions such that a unipotent conjugacy class $c_{\bmu}$ of $\U(n, \FF_q)$ is strongly real if and only if $\bmu(t-1) \in \cP^+$.  Then, a negative unipotent conjugacy class $c_{\bgam}$ of $\U(n, \FF_q)$ is strongly real if and only if $\bgam(t+1) \in \cP^+$.
\end{proposition}
\begin{proof} Let $g$ be a negative unipotent element of $\U(n, \FF_q)$ in the conjugacy class $c_{\bgam}$.  Then $-g$ is a unipotent element in the class $c_{\bmu}$, where $\bmu(t-1) = \bgam(t+1)$.  If $\bgam(t+1) \in \cP^+$, then $\bmu(t-1) \in \cP^+$, so $-g$ is strongly real.  If $s (-g) s = -g^{-1}$ with $s \in \U(n, \FF_q)$ and $s^2 = 1$, then $s g s = g^{-1}$, so $g$ is strongly real.  Conversely, if $g$ is strongly real, then so is $-g$, so $\bmu(t-1) = \bgam(t+1) \in \cP^+$.
\end{proof}

That is, we have the following.

\begin{corollary} \label{UniCor}
The converse of Proposition \ref{oneway} holds for any real $g \in \U(n, \FF_q)$, $q$ odd, if and only if it holds for any unipotent element.
\end{corollary}

Thus, we now focus our attention on unipotent classes of $\U(n, \FF_q)$.  These conjugacy classes are parametrized by partitions of $n$.  For any partition $\mu = (\mu_1, \mu_2, \ldots, \mu_l)$ of $n$, we will call the corresponding unipotent class of $\U(n, \FF_q)$ the unipotent class of \emph{type} $\mu$.  We will use the notation that $m_i$ denotes the multiplicity of the part $i$, and we write $\mu = (k^{m_k} (k-1)^{m_{k-1}} \cdots 2^{m_2} 1^{m_1})$, where $k = \mu_1$ is the largest part of $\mu$. 

For example, if $\mu = (5, 5, 3, 2, 2, 2, 1, 1)$, then we may write $\mu = (5^2 3^1 2^3 1^2)$.  The unipotent class of $\U(21, \FF_q)$ of type $\mu$ consists of elements with elementary divisors
$$(t-1)^5, (t-1)^5, (t-1)^3, (t-1)^2, (t-1)^2, (t-1)^2, t-1, t-1.$$

\section{Reduction of powers} \label{Induction}

Let $V$ be an $n$-dimensional $\FF_{q^2}$-vector space, and let $H$ be a non-degenerate Hermitian form on $V$ which defines a finite unitary group $\U(n, \FF_q)$.  

Throughout this section, let $x$ be a unipotent element of $\U(n, \FF_q)$ of type $\mu = (k^{m_k} \cdots 2^{m_2} 1^{m_1})$.  For each elementary divisor $(t-1)^l$ of $x$, we may choose vectors $v_i \in V$, $i = 1, \ldots, l$, with the property that $x v_1 = v_1$, and $x v_i = v_i + v_{i-1}$ for $i = 2, \ldots, l$.  This follows from the fact that $x \in \U(n, \FF_q)$ is also a unipotent element of type $\mu$ in $\GL(n, \FF_{q^2})$.  If we choose such vectors for every elementary divisor of $x$, label them as follows.  If $(t-1)^l$ has multiplicity $m_l$, label these elementary divisors by $j = 1, 2, \ldots, m_l$.  Let $v_i^{l, j} \in V$, for $1 \leq j \leq m_l$, $1 \leq i \leq l$, be the vectors associated with the elementary divisor $(t-1)^l$ with label $j$, with the property $xv_1^{l, j} = v_1^{l, j}$ and $xv_i^{l, j} = v_i^{l, j} + v_{i-1}^{l, j}$, $2 \leq i \leq l$.  Now, the set of vectors 
$$\{ v_i^{l,j} \; \mid \; 1 \leq l \leq k, 1 \leq j \leq m_l, 1 \leq i \leq l \}$$
forms an $\FF_{q^2}$-linear basis of $V$.  We fix this notation for a basis of $V$ defined by the unipotent element $x \in \U(n, \FF_q)$ of type $\mu = (k^{m_k} \cdots 2^{m_2} 1^{m_1})$ throughout the section.

\begin{lemma} \label{Orthog1} Let $1 \leq l, l' \leq k$, $1 \leq j \leq m_l$, $1 \leq j' \leq m_{l'}$, and $1 \leq r \leq l$.  Then, for any $i, i'$ such that $1 \leq i \leq r$ and $1 \leq i' \leq l'-r$, we have $H(v_i^{l,j}, v_{i'}^{l',j'})=0$.
\end{lemma}
\begin{proof} If $r=1$, the claim is that $H(v_1^{l,j}, v_{i'}^{l',j'}) = 0$ for all $1 \leq i' \leq l'-1$.  Then $2 \leq i'+1 \leq l'$, and since $x \in \U(n, \FF_q)$, then $x$ stabilizes the Hermitian form $H$ by definition.  We then have
\begin{align*}
H(v_1^{l,j}, v_{i'+1}^{l',j'}) & = H(xv_1^{l,j}, xv_{i'+1}^{l',j'})\\
                               & = H(v_1^{l,j}, v_{i'+1}^{l',j'} + v_{i'}^{l',j'})\\
                               & = H(v_1^{l,j}, v_{i'+1}^{l',j'}) + H(v_1^{l,j}, v_{i'}^{l',j'}),
\end{align*}
which gives $H(v_1^{l,j}, v_{i'}^{l',j'}) = 0$.  

Now suppose the statement holds for an $r$, $1 \leq r \leq l-1$, so $H(v_i^{l,j}, v_{i'}^{l',j'})=0$ for any $1 \leq i \leq r$ and $1 \leq i' \leq l'-r$.  To prove that the statement for $r+1$, it is enough to show that for any $1 \leq i' \leq l'-r-1$, we have $H(v_{r+1}^{l,j}, v_{i'}^{l',j'}) = 0$.  Then, for any such $i'$, we have $i' + 1 \leq l'$.  We then have
\begin{align*}
H(v_{r+1}^{l,j}, v_{i'+1}^{l',j'}) & = H(xv_{r+1}^{l,j}, xv_{i'+1}^{l',j'}) \\
                            & = H(v_{r+1}^{l,j} + v_r^{l,j}, v_{i'+1}^{l',j'} + v_{i'}^{l',j'}) \\
                            & = H(v_{r+1}^{l,j}, v_{i'+1}^{l',j'}) + H(v_{r+1}^{l,j}, v_{i'}^{l',j'}) + H(v_{r}^{l,j}, v_{i'+1}^{l',j'}) + H(v_{r}^{l,j}, v_{i'}^{l',j'}).
\end{align*}
From the induction hypothesis, we have $H(v_{r}^{l,j}, v_{i'+1}^{l', j'}) = H(v_{r}^{l,j}, v_{i'}^{l',j'}) = 0$, since $i'+1 \leq l'-r$.  Thus $H(v_{r+1}^{l,j}, v_{i'}^{l',j'})=0$, and by induction, the statement holds for all $1 \leq r \leq l$.
\end{proof}

\begin{lemma} \label{OrthogVec} Let $1 \leq l, l' \leq k$, $1 \leq j \leq m_l$, $1 \leq j' \leq m_{l'}$.  If either $0 < i' \leq l'-1$, or $0 < i' \leq l'$ and $l' < l$, then
 $H(v_1^{l,j}, v_{i'}^{l', j'}) = 0$.
\end{lemma}
\begin{proof}  First, if $0 < i' \leq l'-1$, then the claim follows from Lemma \ref{Orthog1} by taking $r=1$.  So suppose $i' = l' < l$.  Note that we have $H(v_{l'}^{l,j}, v_1^{l',j'})=0$, by the case just covered.  We claim that, in fact, for any $i \leq l'$, we have $H(v_{l'-i+1}^{l,j}, v_{i}^{l',j'}) = 0$.  If the statement holds for an $i$, $1 \leq i \leq l'-1$, then consider
\begin{align*}
H(v_{l'-i+1}^{l,j}, v_{i+1}^{l',j'}) & = H(xv_{l'-i+1}^{l,j}, xv_{i+1}^{l',j'}) \\
                                 & = H(v_{l'-i+1}^{l,j} + v_{l'-i}^{l,j}, v_{i+1}^{l',j'}+ v_i^{l',j'}) \\
& = H(v_{l'-i+1}^{l,j}, v_{i+1}^{l',j'}) + H(v_{l'-i+1}^{l,j}, v_i^{l',j'}) + H(v_{l'-i}^{l,j}, v_{i+1}^{l',j'}) + H(v_{l'-i}^{l,j}, v_i^{l',j'}).
\end{align*}
Now, $H(v_{l'-i+1}^{l,j}, v_i^{l',j'})=0$ by the induction hypothesis, and $H(v_{l'-i}^{l,j}, v_i^{l',j'})=0$ by Lemma \ref{Orthog1} with $r=l'-i$, and so $H(v_{l'-i}^{l,j}, v_{i+1}^{l',j'}) = 0$.  By induction, the statement holds for $i=l'$, and so $H(v_1^{l,j}, v_{l'}^{l',j'}) = 0$.
\end{proof} 

Now, given some $k'$ such that $2 \leq k' \leq k$ and $m_{k'} > 0$, define the subspace $W$ of $V$ as follows:
\begin{equation} \label{Wsub}
W = \mathrm{span} \{ v_1^{l',j'} \, \mid \, k' \leq l' \leq k, 1 \leq j' \leq m_{l'} \} .
\end{equation}

\begin{lemma} \label{xW} Let $2 \leq k' \leq k$ such that $m_{k'} > 0$, and $W$ the subspace as in (\ref{Wsub}).  Then $xW = W$ and $xW^{\perp} = W^{\perp}$, and $W^{\perp}$ is given by
$$ \mathrm{span} \{ v_{i'}^{l', j'} \, \mid \, 1 \leq l' \leq k, 1 \leq j' \leq m_{l'}, 0 < i' \leq l' \text{ and } i' <l' \text{ if } l' \geq k'  \}$$
\end{lemma}
\begin{proof} Letting 
$$U = \mathrm{span} \{v_{i'}^{l',j'} \, \mid \, 1 \leq l' \leq k, 1 \leq j' \leq m_{l'},  0 <i \leq l' \text{ and } i' < l' \text{ if } l' \geq k' \},$$ we have $U \subseteq W^{\perp}$ by Lemma \ref{OrthogVec}.  Since $\mathrm{dim}(U) + \mathrm{dim}(W) = \mathrm{dim}(V)$, then we have $U = W^{\perp}$.  Since $x v_1^{l',j'} = v_1^{l',j'}$ whenever $k' \leq l' \leq k$, $1 \leq j' \leq m_{l'}$, then $xW = W$.  Since $x$ preserves $H$ and $xW=W$, then $xW^{\perp} = W^{\perp}$.  \end{proof}

We now assume that our unipotent element $x$ is strongly real in $\U(n, \FF_q)$, so we let $s \in \U(n,\FF_q)$ have the property that $s^2 = 1$ and $sxs = x^{-1}$.

\begin{lemma} \label{xsbasis}  For any basis element $v_{i'}^{l',j'}$, its image under $xs$ is given by
$$ xsv_{i'}^{l',j'} =  \sum_{i=1}^{i'} (-1)^{i' - i} sv_{i}^{l',j'}.$$
\end{lemma}
\begin{proof}  If $i'=1$, then $xv_1^{l',j'} = v_1^{l',j'} = x^{-1} v_1^{l',j'}$, and since $sxs = x^{-1}$, then $xs v_{1}^{l',j'} = sv_1^{l',j'}$, and the statement holds for $i'=1$.  By induction, suppose the statement holds for an $i' < l'$.  Then since $sxsv_{i'}^{l',j'} = x^{-1} v_{i'}^{l',j'}$, then $x^{-1} v_{i'}^{l',j'} = \sum_{i=1}^{i'} (-1)^{i' -i} v_{i}^{l',j'}$.  Since $i'+1 > 1$, then $xv_{i'+1}^{l',j'} = v_{i'}^{l',j'} + v_{i'+1}^{l',j'}$, so $x^{-1}v_{i'}^{l',j'} + x^{-1}v_{i'+1}^{l',j'} = v_{i'+1}^{l',j'}$.  We now have 
$$sxs v_{i'+1}^{l',j'} = x^{-1} v_{i'+1}^{l',j'} = v_{i'+1}^{l',j'} - x^{-1} v_{i'}^{l',j'} = v_{i'+1}^{l',j'} - \sum_{i=1}^{i'} (-1)^{i'+1-i} v_{i}^{l',j'}.$$
This implies $xsv_{i'+1}^{l',j'} = \sum_{i=1}^{i'+1} (-1)^{i'+1-i}s v_{i}^{l',j'}$, giving the result.
\end{proof}

We next describe the image under $s$ of any $x$-fixed vector.

\begin{lemma} \label{fixvecs}  Suppose $v \in V$ such that $xv = v$.  Then $sv \in \mathrm{span} \{ v_1^{l, j} \, \mid \, 1 \leq l \leq k, 1 \leq j \leq m_l \}$.
\end{lemma}
\begin{proof}  Since $sxs = x^{-1}$ and $xv = v = x^{-1}v$, then $xsv = sv$, or $xsv - sv = 0$.  Write 
$$ sv = \sum_{l = 1}^k \sum_{j=1}^{m_l} \sum_{i=1}^l b_i^{l,j} v_i^{l,j}.$$
Then, from the action of $x$ on each basis element $v_i^{l,j}$, we have
$$ xsv = \sum_{l=1}^k \left( \sum_{j=1}^{m_l} \left( \sum_{i=1}^{l-1} \left( b_i^{l,j} + b_{i+1}^{l,j} \right) v_i^{l,j} \right) + b_l^{l,j} v_l^{l,j} \right).$$
Subtracting $xsv - sv$ gives
\begin{equation} \label{xsv-sv}
xsv - sv = \sum_{l=1}^k \sum_{j=1}^{m_l} \sum_{i=1}^{l-1} b_{i+1}^{l,j} v_i^{l,j}.
\end{equation}
Since also $xsv-sv = 0$, we must then have $b_{i+1}^{l,j}=0$ for all $1 \leq i \leq l-1$, $1 \leq l \leq k$, $1 \leq j \leq m_l$.  That is, $sv = \sum_{l=1}^k \sum_{j=1}^{m_l} b_1^{l,j} v_1^{l,j}$.
\end{proof}

We now introduce some notation for the purpose of understanding the image of any basis vector $v_{i'}^{l',j'}$ under the action of $s$.  Fixing $i', j', l'$, then for any $i, j, l$, define $a_{i, (l',j',i')}^{l,j} \in \FF_{q^2}$ as 
$$ sv_{i'}^{l',j'} = \sum_{l=1}^k \sum_{j=1}^{m_l} \sum_{i=1}^l a_{i, (l',j',i')}^{l,j} v_i^{l,j}.$$
That is, $a_{i, (l',j',i')}^{l,j}$ is the coefficient of $v_i^{l,j}$ in the vector $sv_{i'}^{l',j'}$.

\begin{lemma} \label{s-action}
For any basis element $v_{i'}^{l',j'}$, we have
$$ sv_{i'}^{l',j'} = \sum_{l=1}^k \sum_{j=1}^{m_l} \sum_{i=1}^{\mathrm{min}\{i',l\}} a_{i, (l',j',i')}^{l,j} v_i^{l,j}.$$
\end{lemma}
\begin{proof} Consider the case $i' = 1$.  Then $x v_1^{l',j'} = v_1^{l',j'}$, and so by Lemma \ref{fixvecs} we have
$$ sv_1^{l',j'} = \sum_{l=1}^k \sum_{j=1}^{m_l} a_{1, (l',j',1)}^{l,j} v_1^{l,j} = \sum_{l=1}^k \sum_{j=1}^{m_l} \sum_{i=1}^{ \mathrm{min}\{1, l\}} a_{i, (l',j',1)}^{l,j} v_i^{l,j} .$$
That is, the result holds for $i'=1$ by Lemma \ref{fixvecs}.  By induction, suppose that for an $r < l'$, the statement holds for all $i'$, $1 \leq i' \leq r$, and we consider $v_{r+1}^{l',j'}$.  By Lemma \ref{xsbasis}, we have $xsv_{r+1}^{l',j'} - sv_{r+1}^{l',j'} = \sum_{\tau=1}^r (-1)^{r-\tau+1} s v_{\tau}^{l',j'}$, and by the induction hypothesis, for each $\tau$, $1 \leq \tau \leq r$, we have 
$$ sv_{\tau}^{l',j'} = \sum_{l=1}^k \sum_{j=1}^{m_l} \sum_{i=1}^{\mathrm{min}\{ \tau, l \}} a_{i, (l',j', \tau)}^{l,j} v_i^{l,j}.$$
Also, just as in the calculation in the proof of Lemma \ref{fixvecs} resulting in (\ref{xsv-sv}), we have
$$ xsv_{r+1}^{l',j'} - sv_{r+1}^{l',j'} = \sum_{l=1}^k \sum_{j=1}^{m_l} \sum_{i=1}^{l-1} a_{i+1, (l',j', r+1)}^{(l,j)} v_i^{l,j}.$$
Equating the two resulting expressions for $xsv_{r+1}^{l',j'} -sv_{r+1}^{l',j'}$ yields
\begin{equation} \label{twosides}
\sum_{l=1}^k \sum_{j=1}^{m_l} \sum_{i=1}^{l-1} a_{i+1, (l',j', r+1)}^{l,j} v_i^{l,j} = \sum_{\tau=1}^r (-1)^{r - \tau + 1} \sum_{l=1}^k \sum_{j=1}^{m_l} \sum_{i=1}^{\mathrm{min}\{ \tau, l \}} a_{i, (l',j', \tau)}^{l,j} v_i^{l,j}.
\end{equation}
On the right side of (\ref{twosides}), every vector $v_i^{l,j}$ appearing in the sum satisfies $i \leq r$.  So, from the left side, we must have $a_{i+1, (l',j', r+1)}^{l,j} = 0$ whenever $i > r$, $i \leq l-1$.  That is, $a_{i, (l',j',r+1)}^{l,j} \neq 0$ only if $i \leq \mathrm{min}\{r+1, l\}$, as claimed.
\end{proof}

In the case of a basis vector of the type $v_1^{l',j'}$, we may strengthen Lemma \ref{s-action} in the following form.

\begin{lemma} \label{s-action1} For any basis vector of the form $v_{1}^{l',j'}$, its image under $s$ is given by
$$ sv_1^{l',j'} = \sum_{l=l'}^k \sum_{j=1}^{m_l} a_{1, (l',j',1)}^{l,j} v_1^{l,j}.$$
\end{lemma}
\begin{proof}  Consider a general basis element $v_{i'}^{l',j'}$, such that $1 < i' < l'$.  Then, by Lemma \ref{s-action}, and the same calculation as in the proof of Lemma \ref{s-action} giving (\ref{twosides}), replace $r+1$ with $i'$ to obtain the following equation, where each side is equal to $xsv_{i'}^{l',j'} - sv_{i'}^{l',j'}$:
\begin{equation} \label{twosidesi'}
\sum_{l=1}^k \sum_{j=1}^{m_l} \sum_{i=1}^{l-1} a_{i+1, (l',j', i')}^{l,j} v_i^{l,j} = \sum_{\tau=1}^{i'-1} (-1)^{i' - \tau} \sum_{l=1}^k \sum_{j=1}^{m_l} \sum_{i=1}^{\mathrm{min}\{ \tau, l \}} a_{i, (l',j', \tau)}^{l,j} v_i^{l,j}.
\end{equation}
For any $l$ such that $i'-1 \leq l$, and any $j \leq m_l$, consider the coefficient of $v_{i'-1}^{l,j}$ on each side of (\ref{twosidesi'}).  This gives $a_{i', (l',j',i')}^{l,j} = -a_{i'-1,(l',j',i'-1)}^{l,j}$.  In particular, if $l=i'$, then for any $j \leq m_{i'}$, we have $a_{i', (l',j',i')}^{i',j} = -a_{i'-1, (l',j',i'-1)}^{i',j}$.  Since this holds for any $i'$ such that $1 < i' <l'$, we have 
\begin{equation} \label{subtract}
a_{i',(l',j',i')}^{i',j} = -a_{i'-1,(l',j',i'-1)}^{i',j} = a_{i'-2,(l',j',i'-2)}^{i',j} = \cdots = (-1)^{i'-1} a_{1,(l',j',1)}^{i',j}.
\end{equation}

Now consider the vector $v_{i'+1}^{l',j'}$, and the equation obtained like that in (\ref{twosidesi'}), by considering the two expressions for $xsv_{i'+1}^{l',j'} - sv_{i'+1}^{l',j'}$.  This is
\begin{equation} \label{twosidesi'1}
\sum_{l=1}^k \sum_{j=1}^{m_l} \sum_{i=1}^{l-1} a_{i+1, (l',j', i'+1)}^{l,j} v_i^{l,j} = \sum_{\tau=1}^{i'} (-1)^{i' + 1 - \tau} \sum_{l=1}^k \sum_{j=1}^{m_l} \sum_{i=1}^{\mathrm{min}\{ \tau, l \}} a_{i, (l',j', \tau)}^{l,j} v_i^{l,j}.
\end{equation}
Consider the coefficient of $v_{i'}^{i',j}$, for any $j \leq m_{i'}$, on both sides of (\ref{twosidesi'1}).  This vector does not appear on the left side of (\ref{twosidesi'1}), while the coefficient on the right side is $a_{i',(l',j',i')}^{i',j}$, and so $a_{i',(l',j',i')}^{i',j} = 0$.  We note that this holds also for $i'=1$ from this argument.  For $i' > 1$, from (\ref{subtract}), we have $a_{i,(l',j',i)}^{i',j}=0$ for any $i \leq i'$.  In particular, letting $i'=l$, then for any $l < l'$ and any $j \leq m_l$, we have $a_{1,(l',j',1)}^{l,j} = 0$, as desired.
\end{proof}

\begin{lemma} \label{sW} Let $s \in \U(n, \FF_q)$ be such that $s^2 = 1$ and $sxs = x^{-1}$ for the unipotent element $x$, and for some $k'$ such that $2 \leq k' \leq k$ and $m_{k'} > 0$, let $W \subseteq V$ be the subspace as in (\ref{Wsub}).  Then $sW = W$ and $sW^{\perp} = W^{\perp}$.
\end{lemma}
\begin{proof}  Let $v_1^{l',j'}$ be any basis element of $W$ as in (\ref{Wsub}).  By Lemma \ref{s-action1}, $sv_1^{l',j'}$ is a linear combination of other such basis elements, hence $sW = W$.  Since $s$ preserves the Hermitian form $H$ and $sW=W$, it follows that $sW^{\perp} = W^{\perp}$ as well.
\end{proof}

We may now state and prove the main result of this section, which is the crucial step for the induction proof of the main result of the paper.

\begin{proposition} \label{IndStep} Suppose the unipotent class of type $\mu$ of $\U(n, \FF_q)$ is strongly real, where 
$$\mu = \left(k^{m_k} (k-1)^{m_{k-1}} \cdots l^{m_l} (l-1)^{m_{l-1}} \cdots 2^{m_2} 1^{m_1} \right),$$
with $m_l > 0$, $l \geq 2$.  Consider the unipotent class of type $\mu^{\#}$, where
$$\mu^{\#} = \left( (k-2)^{m_k} \cdots (l+1)^{m_{l+3}} l^{m_{l+2}} (l-1)^{m_{l-1} + m_{l+1}} (l-2)^{m_{l-2} + m_l} (l-3)^{m_{l-3}} \cdots \right), $$
if $3 \leq l < k$, or
$$\mu^{\#} = \left( (k-2)^{m_k} \cdots 2^{m_4} 1^{m_1 + m_3} \right),$$
if $l = 2$, or
$$\mu^{\#} = \left( (k-1)^{m_{k-1}} (k-2)^{m_{k-2}+ m_{k}} (k-3)^{m_{k-3}} \cdots 1^{m_1}\right),$$
if $l =k$.  Then the unipotent class $\mu^{\#}$ of $\U(n^{\#}, \FF_q)$, where $n^{\#} = n - 2\sum_{l' \geq l} m_{l'}$, is strongly real.
\end{proposition}
\begin{proof}  We are assuming the unipotent element $x \in \U(n, \FF_q)$ of type $\mu$ is strongly real, so let $s \in \U(n, \FF_q)$ be such that $s^2 = 1$ and $sxs = x^{-1}$.  Given $l$, $2 \leq l \leq k$, define $W \subseteq V$ as in (\ref{Wsub}) with $k'=l$.  By the basis for $W^{\perp}$ given in Lemma \ref{xW} and the definition of $W$, we have $W \subset W^{\perp}$.  Consider the quotient space $W^{\perp}/W$.  From this definition, and from Lemma \ref{xW}, a basis for $W^{\perp}/W$ is given by
$$ \{v_{i'}^{l',j'} + W \, \mid \, 1 \leq l' < l, 1 \leq j' \leq m_{l'}, 0 < i' \leq l'\} \cup \{ v_{i'}^{l',j'} + W \, \mid \, l \leq l' \leq k, 1 \leq j' \leq m_{l'}, 1 < i' < l'\}.$$
In particular, this gives 
$$\mathrm{dim}(W^{\perp}/W) = n - 2 \sum_{l' \geq l} m_{l'} = n^{\#}.$$

Define a Hermitian form induced by $H$, $H_1$, on $W^{\perp}/W$ by $H_1(u + W, v + W) = H(u,v)$.  Since $xW = W$ and $xW^{\perp} = W^{\perp}$ by Lemma \ref{xW}, then $x$ has an induced action of $W^{\perp}/W$, giving an element $x_1$ defined by $x_1(v + W) = xv + W$, where $x_1$ is in the finite unitary group $\U(n^{\#}, \FF_q)$ defined by the Hermitian form $H_1$.  From the action of $x_1$ on the basis for $W^{\perp}/W$ above, it follows that $x_1$ is a unipotent element of $\U(n^{\#}, \FF_q)$, where the partition describing the class of $x_1$ is obtained from that of $x$ by subtracting $2$ from every part of $\mu$ which is $l$ or larger.  That is, $x_1$ is a unipotent element of type $\mu^{\#}$.

Now, since $s$ also satisfies $sW = W$ and $sW^{\perp} = W^{\perp}$, there is an induced action of $s$ on $W^{\perp}/W$, giving an element $s_1$ in $\U(n^{\#}, \FF_q)$ defined by $H_1$.  It follows immediately that $s_1^2 = 1$ and $s_1 x_1 s_1 = x_1^{-1}$, and so $x_1$ is strongly real in $\U(n^{\#}, \FF_q)$.
\end{proof}

For example, Proposition \ref{IndStep} says that if a unipotent element of type $\mu = (8^5 6^4 5^2 4^1 3^2 2^8 1^3)$ is strongly real, then so is a unipotent element of type $(6^5 4^4 3^2 2^1 1^5)$, by taking $l = 2$, and so is a unipotent element of type $(6^5 4^4 3^4 2^9 1^3)$, by taking $l = 4$.

\section{Main Results} \label{Main}

\subsection{Proof of the main theorem}

The following result is a base case for the induction argument in the proof of the main theorem.

\begin{lemma} \label{TwoLemma} Let $q$ be odd.  Then every unipotent class in $\U(n, \FF_q)$ of type $(2^{m_2} 1^{m_1})$, with $m_2$ odd, is not strongly real.
\end{lemma}
\begin{proof} For notational convenience, let $m_2 = r$ and $m_1 = m$, so $n = 2r + m$ and $r$ is odd.

As in (\ref{Umatrix}), we are free to choose the matrix $J$ defining the Hermitian form for the finite unitary group.  For any positive integer $d$, let $I_d$ denote the $d$-by-$d$ identity matrix, and let $N_d$ denote the $d$-by-$d$ matrix with $1$'s on the anti-diagonal, and $0$'s elsewhere, that is,
$$ N_d = \left[ \begin{array}{ccc}  &   &  1 \\   &  \adots &  \\ 1  &   &   \end{array} \right].$$
Then we define $\U(n, \FF_q)$ by the Hermitian form corresponding to the matrix
$$ J = \left[ \begin{array}{cc} N_{2r} &  \\   &  I_m  \end{array} \right] = \left[ \begin{array}{ccc}  & N_r  &   \\  N_r  &   &   \\  &   &  I_m \end{array} \right].$$
It is enough to show that any particular unipotent element of type $(2^r 1^m)$ in $\U(n, \FF_q)$ is not strongly real.  One such element is
$$ g = \left[ \begin{array}{ccc} I_r & aI_r &  \\  & I_r &   \\  &   &  I_m \end{array} \right],$$
where $a \in \FF_{q^2}$ such that $a + \bar{a} = 0$ and $a \neq 0$.  Suppose that $g$ is strongly real, so that there is some $s \in \U(n,\FF_q)$ such that $s^2 = 1$ and $sgs = g^{-1}$.  Writing $s$ in block form, a direct computation gives that the condition $sgs=g^{-1}$ forces $s$ to be of the form
$$ s = \left[ \begin{array}{ccc} s_{11} & s_{12} & s_{13} \\  & -s_{11} &  \\   & s_{32} & s_{33} \end{array} \right],$$
where $s_{11}$ and $s_{12}$ are $r$-by-$r$, $s_{13}$ is $r$-by-$m$, $s_{32}$ is $m$-by-$r$, and $s_{33}$ is $m$-by-$m$.  From the condition $s^2 = 1$, it follows that $s_{11}^2 = 1$, so that $\mathrm{det}(s_{11}) = \pm 1$.  Further, the condition $s \in \U(n, \FF_q)$, so that ${^\top \bar{s}} J s = J$, implies that we must have ${^\top \overline{s_{11}}} N_r s_{11} = -N_r$.  Taking the determinant of both sides of the last equation gives $1 = (-1)^r$.  Since $r$ and $q$ are both assumed to be odd, this is a contradiction, and so the element $g$ of type $(2^r 1^m)$ is not strongly real.  \end{proof}

Finally, we arrive at the main result.

\begin{theorem} \label{MainThm}  Let $q$ be odd.  A real element $g$ of $\U(n, \FF_q)$ is strongly real if and only if every elementary divisor of the form $(t\pm 1)^{2m}$ of $g$ has even multiplicity.
\end{theorem}
\begin{proof}
By Proposition \ref{oneway}, we know that if a real element of $\U(n, \FF_q)$ satisfies these conditions, then it is strongly real.  By Corollary \ref{UniCor}, it is enough to prove the converse statement for unipotent elements.

So, we consider a unipotent element of type $\mu = (k^{m_k} \cdots 1^{m_1})$.  We first show that if $m_2$ is odd, while $m_{2i}$ is even for any $i > 1$, then the unipotent class of type $\mu$ is not strongly real.  We prove this claim by induction on the largest part $k$ of $\mu$.  For $k=2$, the statement is exactly Lemma \ref{TwoLemma}.  So suppose the statement holds for any $k < i$, where $i \geq 3$, and consider $\mu = (i^{m_j} \cdots 2^{m_2} 1^{m_1})$, where $m_i \neq 0$.  If an element of type $\mu$ is strongly real, then we may apply Proposition \ref{IndStep}, and consider three cases.  If $i=3$, so $\mu = (3^{m_3} 2^{m_2} 1^{m_1})$, then a unipotent element of type $\mu$ cannot be strongly real, otherwise unipotent elements of type $\mu^{\#} = (2^{m_2} 1^{m_1 + m_3})$ are strongly real, contradicting Lemma \ref{TwoLemma}.  If $i=4$, then $\mu = (4^{m_4} 3^{m_3} 2^{m_2} 1^{m_1})$, where $m_4$ is even by assumption.  Then unipotent elements of type $\mu$ cannot be strongly real, otherwise unipotent elements of type $\mu^{\#} = (2^{m_2 + m_4} 1^{m_1 + m_3})$ are strongly real, where $m_2 + m_4$ is odd, again contradicting Lemma \ref{TwoLemma}.  Finally, if $i>4$, with
$$ \mu = \left(i^{m_i} (i-1)^{m_{i-1}} (i-2)^{m_{i-2}} \cdots 2^{m_2} 1^{m_1}\right),$$
then if unipotent elements of type $\mu$ are strongly real, so are unipotent elements of type
$$\mu^{\#} = \left((i-1)^{m_{i-1}} (i-2)^{m_{i-2} + m_i} (i-3)^{m_{i-3}} \cdots 2^{m_2} 1^{m_1} \right),$$
where $m_{i-2} + m_i \neq 0$, and is even if $i$ is even, since $i >2$.  This contradicts the induction hypothesis, giving the claim.

Now, with $\mu =  (k^{m_k} \cdots 2^{m_2} 1^{m_1})$, we must show that if there exists some even $j$ such that $m_j$ is odd, then the unipotent element of this type is not strongly real.  We prove this by induction on the largest such $j$.  The statement for $j=2$ is exactly the claim just proved above.  Assume it holds true for all even $j < i$, for some even $i \geq 4$.

Consider a class of type $\mu = (k^{m_k} \cdots 2^{m_2} 1^{m_1})$, where the largest even part with odd multiplicity is $i \geq 4$.  If the class of type $\mu$ is strongly real, then by Proposition \ref{IndStep}, the class of type $\mu^{\#}$, where 
$$\mu^{\#} = \left( (k-2)^{m_k} \cdots j^{m_{j+2}} (j-1)^{m_{j+1}} (j-2)^{m_j} \cdots 2^{m_4} 1^{m_1 + m_3} \right),$$
is also strongly real.  However, the largest even part of $\mu^{\#}$ with odd multiplicity is $j-2$, since we know $m_w$ is even for any even $w > j$.  This contradicts the induction hypothesis, completing the proof.
\end{proof}

\subsection{Symplectic groups}

Let $q$ be odd.  If $V$ is a $2n$-dimensional $\FF_q$-vector space, with a non-degenerate alternating bilinear form, then this form may be extended to a non-degenerate Hermitian form on the $2n$-dimensional $\FF_{q^2}$-vector space obtained by extending $V$ by scalars.  In other words, the finite unitary group $\U(2n, \FF_q)$ contains the finite symplectic group $\Sp(2n, \FF_q)$ as a subgroup.  Here, we obtain partial results on strongly real conjugacy classes in $\Sp(2n, \FF_q)$ by applying this fact and Theorem \ref{MainThm}. 

The conjugacy classes of the finite symplectic group are described by Wall \cite[Sec. 2.6, Case (B)]{Wa62} (see also \cite{Fu00} for a concise description).  Using these results, we can describe the conjugacy classes of $\Sp(2n, \FF_q)$ in terms of the classes of $\U(2n, \FF_q)$ as follows.  Given a conjugacy class of $\U(2n, \FF_q)$ labeled by $\bmu$, then this class contains elements in $\Sp(2n, \FF_q)$ if and only if the following hold:

\begin{enumerate}

\item  The class $c_{\bmu}$ is a real class of $\U(2n, \FF_q)$, that is, $\bmu(f) = \bmu(\tilde{f})$ for every $f \in \cU$.

\item  For $f = t \pm 1$, every odd part of the partition $\bmu(f)$ has even multiplicity.

\end{enumerate}

Given a class of $\U(2n, \FF_q)$ which satisfies the above, we may further describe precisely how it splits into distinct classes in $\Sp(2n, \FF_q)$.  In particular, if $c_{\bmu}$ is such a class, and $k_1$ and $k_2$ are the number of even integers with nonzero multiplicity in the partitions $\bmu(t-1)$ and $\bmu(t+1)$ respectively, then $c_{\bmu}$ splits into $2^{k_1 + k_2}$ distinct classes in $\Sp(2n, \FF_q)$.

This fact prompts the following definition.  Given a partition $\lambda$, let $m_j(\lambda)$ denote the multiplicity of the part $j$ in $\lambda$.  Define a \emph{symplectic signed partition} to be a partition $\lambda$, satisfying $m_j(\lambda)$ is even whenever $j$ is odd, together with a function
$$ \delta: \{ 2i > 0 \, \mid \, m_{2i}(\lambda) \neq 0 \} \longrightarrow \{ \pm 1 \},$$
which assigns a sign to each even integer which has nonzero multiplicity in $\lambda$.  Let $\cP^{\pm}$ denote the set of all symplectic signed partitions.  For example, 
$$ \gamma = (5, 5, 4^-, 4^-, 3, 3, 2^+, 2^+, 2^+, 1, 1, 1, 1) = (5^2, 4^{-2}, 3^2, 2^{+3}, 1^4)$$
is a symplectic signed partition, where the parts of size $4$ are assigned the sign $-$ and the parts of size $2$ are assigned the sign $+$.  Given $\gamma \in \cP^{\pm}$, let $\gamma^{\circ}$ denote the partition obtained by ignoring the sign function $\delta$.  So, given $\gamma$ in the example above, we have $\gamma^{\circ} = (5^2, 4^2, 3^2, 2^3, 1^4)$.

Summarizing all of this, we have the following.

\begin{proposition} \label{Spclasses}
The conjugacy classes of $\Sp(2n, \FF_q)$, $q$ odd, are parametrized by functions
$$\bmu:  \cU \rightarrow \cP \cup \cP^{\pm},$$
such that $\bmu(f) \in \cP$ and $\bmu(f) = \bmu(\tilde{f})$ if $f \neq t \pm 1$, $\bmu(t\pm 1) \in \cP^{\pm}$, and 
$$|\bmu(t+1)^{\circ}| + |\bmu(t-1)^{\circ}| + \sum_{f \in \cU, f \neq t \pm 1} d(f) |\bmu(f)| = 2n.$$
The conjugacy class of $\U(2n, \FF_q)$ to which $\bmu$ corresponds is the class $c_{\bmu^{\circ}}$, where $\bmu^{\circ}: \cU \rightarrow \cP$ is defined by $\bmu^{\circ}(f) = \bmu(f)$ if $f \neq t \pm 1$, and $\bmu^{\circ}(t \pm 1) = \bmu(t \pm 1)^{\circ}$.
\end{proposition}

From Theorem \ref{MainThm} and Proposition \ref{Spclasses}, we may immediately conclude the following.

\begin{corollary} \label{SpCor}
Let $q$ be odd, and consider a conjugacy class of $\Sp(2n, \FF_q)$ parametrized by $\bmu: \cU \rightarrow \cP \cup \cP^{\pm}$.  If, for some $2i > 0$, either $m_{2i}(\bmu(t-1)^{\circ})$ or $m_{2i}(\bmu(t+1)^{\circ})$ is odd, then this conjugacy class is not strongly real in $\Sp(2n, \FF_q)$.
\end{corollary}

We note that there does not yet seem to be a classification of the strongly real classes of the group $\Sp(2n, \FF_q)$, $q$ odd, and so Corollary \ref{SpCor} is a start in this direction.  It is known, however, that when $q \equiv 1($mod $4)$, then all classes of $\Sp(2n, \FF_q)$ are real, while there are classes which are not real when $q \equiv 3($mod $4)$, see \cite{Wo66, FeZu82}.

\subsection{Enumeration of strongly real classes}
Consider a conjugacy class of $\U(n, \FF_q)$ parametrized by $\bmu$, as in Section \ref{ConjClasses}.  For each positive integer $i$, define a polynomial $u_i(t)$ by
$$ u_i(t) = \prod_{f \in \cU} f(t)^{m_i(\bmu(f))}.$$
Define a partition $\nu$ of $n$ as follows.  For each positive integer $i$, define
$n_i = \sum_{f \in \cU} d(f) m_i(\bmu(f))$, and let $n_i = m_i(\nu)$.  That is, $\nu = (\cdots 3^{n_3} 2^{n_2} 1^{n_1})$, so that $\nu$ is a partition of $n$.  We then have $n_i = d(u_i)$, and $\prod_i u_i(t)^i$ is the characteristic polynomial of the elements in the conjugacy class parametrized by $\bmu$.  Conversely, given any sequence of polynomials $(u_1(t), u_2(t), \ldots)$, each of which is a product of elements in $\cU$, such that $\sum_i i d(u_i) = n$, we may recover a unique $\bmu$ by factoring each $u_i(t)$ into elements of $\cU$, which is a unique factorization by \cite{En62}.  Thus, such a sequence of polynomials in $\FF_{q^2}[t]$ also parametrizes conjugacy classes in $\U(n, \FF_q)$.

Now consider a real conjugacy class of $\U(n, \FF_q)$, corresponding to the sequence of polynomials $(u_1(t), u_2(t), \ldots)$ as above.  By the discussion in Section \ref{RealClasses}, each $u_i(t)$ satisfies $\tilde{u_i} = u_i$, that is, are \emph{self-conjugate} monic polynomials, and $u_i(t) \in \FF_q[t]$.  So, real conjugacy classes of $\U(n, \FF_q)$ (and of $\GL(n, \FF_q)$, as we mentioned in Section \ref{RealClasses}), are parametrized by $(u_1(t), u_2(t), ...)$ such that each $u_i(t) \in \FF_q[t]$ is self-conjugate.  Consider the factorization of each $u_i(t)$ into products of irreducible polynomials in $\FF_q[t]$.  Since $u_i(t)$ is self-conjugate, each of its irreducible factors must either be self-conjugate itself, or occur with equal power with its conjugate in $\FF_q[t]$.  By \cite[Lemma 1.3.15]{FuNePr05}, the only irreducible self-conjugate polynomials in $\FF_q[t]$ of odd degree are $t-1$ and $t+1$, and every monic self-conjugate polynomial has constant $\pm 1$.  Then, the constant term of $u_i(t)$ is $1$ if and only if the power of $t-1$ in its factorization is even, and if this power is even, then the degree of $u_i(t)$ is even if and only if the power of $t+1$ in the factorization of $u_i(t)$ is even.  Since the powers of $t\pm 1$ in $u_i(t)$ are precisely $m_i(\bmu(t \pm 1))$, then we may rephrase Theorem \ref{MainThm} as follows.

\begin{corollary} \label{MainRephrase} Let $(u_1(t), u_2(t), \ldots)$ correspond to a conjugacy class of $\U(n, \FF_q)$, with $q$ odd.  Then this class is strongly real if and only if each $u_i(t)$ is a monic self-conjugate polynomial in $\FF_q[t]$ with nonzero constant, and whenever $i$ is even, $u_i(t)$ has even degree and constant term $1$.
\end{corollary}

We may now use Corollary \ref{MainRephrase} to enumerate the strongly real classes in $\U(n, \FF_q)$ when $q$ is odd.  To specify a strongly real class of $\U(n, \FF_q)$, we choose $(u_1(t), u_2(t), \ldots)$ such that, when $i$ is odd, $u_i(t)$ is any self-conjugate polynomial in $\FF_q[t]$ with non-zero constant, of degree $n_i$, and when $i$ is even, $u_i(t)$ is a self-conjugate polynomial in $\FF_q[t]$ with constant $1$ of even degree $n_i$, such that $\sum_i i n_i = n$.  By \cite[Lemma 1.3.15]{FuNePr05}, the number of self-conjugate polynomials in $\FF_q[t]$ of degree $n_i \geq 1$ with non-zero constant is $q^{\lfloor n_i/2 \rfloor} + q^{\lfloor (n_i -1)/2 \rfloor}$, and the number of self-conjugate polynomials in $\FF_q[t]$ with constant $1$ of even degree $n_i$ is $q^{n_i/2}$.  This allows us to compute the following.

\begin{corollary} \label{enumeration} Let $T_{n,q}$ be the number of strongly real classes in $\U(n, \FF_q)$, where $q$ is odd (where $n=0$ gives the trivial group).  Then a generating function for $T_{n,q}$ is given by
$$ \sum_{n=0}^{\infty} T_{n,q} z^n = \prod_{k=1}^{\infty} \frac{(1 + qz^{2k-1})^2}{1 - qz^{2k}}.$$
\end{corollary}
\begin{proof}  From the discussion above, we know that
$$T_{n,q} = \sum_{|\nu| = n \atop{ \nu=(\cdots 2^{n_2} 1^{n_1})}} c_{n_1, q} c_{n_2, q} \cdots,$$
where $c_{n_i,q} = q^{\lfloor n_i/2 \rfloor} + q^{\lfloor (n_i -1)/2 \rfloor}$ when $i \geq 1$ is odd, and when $i$ is even, $c_{n_i, q} = q^{n_i/2}$ if $n_i$ is even, and $c_{n_i, q} = 0$ if $n_i$ is odd.  So,
\begin{align} \label{ProdTerms}
\sum_{n=0}^{\infty} T_{n,q} z^n & = \sum_{n=0}^{\infty} \left(\sum_{|\nu|=n \atop{\nu=( \cdots 2^{n_2} 1^{n_1})}} c_{n_1, q} c_{n_2, q} \cdots \right) z^n  \notag\\
& = \sum_{n=0}^{\infty} \left[ \sum_{m=0}^n \left( \sum_{|\sigma|=m \atop{\sigma = ( \cdots 4^{n_4} 2^{n_2})} } c_{n_2, q} c_{n_4, q} \cdots \right) \left( \sum_{|\gamma| = m-n \atop{\gamma = (\cdots 3^{n_3} 1^{n_1})}} c_{n_1, q} c_{n_3, q} \cdots \right) \right] z^n \notag\\
& = \left[ \sum_{n=0}^{\infty} \left( \sum_{|\sigma|=m \atop{\sigma = (\cdots 4^{n_4} 2^{n_2} )}} c_{n_2, q} c_{n_4, q} \cdots \right) z^n \right] \left[ \sum_{n=0}^{\infty} \left( \sum_{|\gamma| = n \atop{\gamma = (\cdots 3^{n_3} 1^{n_1})}} c_{n_1, q} c_{n_3, q} \cdots \right) z^n \right]
\end{align}
For the first term in (\ref{ProdTerms}), we have
$$\sum_{n=0}^{\infty} \left( \sum_{|\sigma|=m \atop{\sigma = (\cdots 4^{n_4} 2^{n_2})}} c_{n_2, q} c_{n_4, q} \cdots \right) z^n = \prod_{k \text{ even}} \sum_{j=0}^{\infty} A_j z^{kj},$$
where $A_j = 0$ if $j$ is odd, and $A_j = q^{j/2}$ if $j$ is even.  So,
\begin{equation} \label{FirstFactor}
\prod_{k \text{ even}} \sum_{j=0}^{\infty} A_j z^{kj}  = \prod_{k \text{ even}} \sum_{i=0}^{\infty} q^i z^{2ki} = \prod_{k \text{ even}} \frac{1}{1 - qz^{2k}}.
\end{equation}
For the second term in (\ref{ProdTerms}), we have
$$\sum_{n=0}^{\infty} \left( \sum_{|\gamma| = n \atop{\gamma = (\cdots 3^{n_3}1^{n_1} )}} c_{n_1, q} c_{n_3, q} \cdots \right) z^n = \prod_{k \text{ odd}} \sum_{j=0}^{\infty} B_j z^{kj},$$
where $B_j = q^{\lfloor j/2 \rfloor} + q^{\lfloor (j-1)/2 \rfloor}$ if $j \geq 1$, and $B_0 = 1$, so $B_j = 2 q^{(j-1)/2}$ if $j$ is odd, and $B_j = q^{j/2} + q^{j/2 - 1}$ if $j \geq 2$ is even.  Then we have
\begin{align} \label{SecondFactor} 
\prod_{k \text{ odd}} \sum_{j=0}^{\infty} B_j z^{kj} & = \prod_{k \text{ odd}} \left( 1 + \sum_{i = 0}^{\infty} 2 q^i z^{(2i+1)k} + \sum_{i=1}^{\infty} (q^{i+1} + q^i) z^{2ik} \right)  \notag\\ 
& = \prod_{k \text{ odd}} \left( 2qz^k \sum_{i = 0}^{\infty} q^i z^{2ik} + \sum_{i=0}^{\infty} q^i z^{2ik} + q \sum_{i = 0}^{\infty} q^i z^{2ik} - q \right) \notag \\
& = \prod_{k \text{ odd}} \left( \frac{1 + q + 2qz^k}{1 - qz^{2k}} - q \right) = \prod_{k \text{ odd}} \frac{(1 + qz^k)^2}{1 - qz^{2k}}.
\end{align}
Plugging (\ref{FirstFactor}) and (\ref{SecondFactor}) back into (\ref{ProdTerms}) gives
$$\sum_{n=0}^{\infty} T_{n,q} z^n = \prod_{k \text{ even}} \frac{1}{1 - qz^{2k}} \prod_{k \text{ odd}} \frac{(1 + qz^k)^2}{1 - qz^{2k}},$$
which yields the desired generating function.
\end{proof}

Let $K_{n,q}$ be the total number of conjugacy classes in $\U(n, \FF_q)$, and let $R_{n,q}$ be the number of real classes in $\U(n, \FF_q)$ with $q$ odd.  In contrast with Corollary \ref{enumeration}, it is a result of Wall \cite[Sec. 2.6, Case (A), part (iii)]{Wa62} that $\sum_{n=0}^{\infty} K_{n,q} z^n = \prod_{k=1}^{\infty} \frac{1 + z^k}{1-qz^k}$, and it follows from the calculation (\ref{SecondFactor}) above that $\sum_{n=0}^{\infty} R_{n,q} z^n = \prod_{k=1}^{\infty} \frac{(1 + qz^k)^2}{1 - qz^{2k}}$.

\section{Remarks on characteristic $2$} \label{CharTwo}

In this section, we give some partial results on the strongly real classes of $\U(n, \FF_q)$ in the case that $q$ is a power of $2$.  Throughout this section, we fix $q$ to be a power of $2$.  It is useful that the results in Sections \ref{Unipotent} and \ref{Induction} apply to the characteristic $2$ case, but these are not enough to give a complete classification.  The starting point is Proposition \ref{onewayeven}, which gives a collection of elements in $\U(2n, \FF_q)$ which are strongly real.  Our first result shows that the converse of Proposition \ref{onewayeven} is false, unlike the converse of Proposition \ref{oneway}.

\begin{proposition} \label{31} A unipotent element of type $(3,1)$ in $\U(4, \FF_q)$, $q$ even, is strongly real.
\end{proposition}
\begin{proof}  Define the matrix $N_d$ as in the proof of Lemma \ref{TwoLemma}.  Define $\U(4, \FF_q)$ by the Hermitian form corresponding to the matrix $J = \left[ \begin{array}{cc} N_3 &  \\  & 1 \end{array} \right]$.  One unipotent element in $\U(4, \FF_q)$ of type $(3,1)$ is then an element of the form $g = \left[ \begin{array}{cccc} 1 & a & b &  \\   & 1 & \bar{a} &   \\   &   & 1 &  \\   &  &  & 1 \end{array} \right]$, where $b + \bar{b} = a \bar{a}$, and $a \neq 0$.  We may assume that the polynomial $t^2 + t + 1 \in \FF_q[t]$ is reducible over $\FF_{q^2}$.  Let $\beta$ be a zero of this polynomial in $\FF_{q^2}$, and let $\alpha = \beta a$.  Then we have $\alpha^2 + a \alpha + a^2 = 0$.

Consider the element $s = \left[ \begin{array}{cccc} 1 & \alpha &   & \alpha \\   & 1 & \bar{\alpha} &    \\   &  & 1 &  \\  &   &  \bar{\alpha} & 1 \end{array} \right]$.  Then ${^\top \bar{s}} J s = J$, so $s \in \U(4, \FF_q)$, $s^2 = 1$, and $s g s = g^{-1}$, which follows from the fact that $\alpha^2 + a \alpha + a^2 = \bar{\alpha} + a \alpha + \bar{a} = 0$.  Thus $g$ is a strongly real element of $\U(4, \FF_q)$. 
\end{proof}

By combining Proposition \ref{onewayeven} with Proposition \ref{31}, and taking direct sums of elements in smaller unitary groups, we obtain the following statement, which gives a larger set of strongly real elements in $\U(n, \FF_q)$.

\begin{proposition} \label{Real2} If an element in $\U(n, \FF_q)$, $q$ even, has the property that either every elementary divisor of the form $(t-1)^{2m+1}$, $m \geq 1$, has even multiplicity, or has $t-1$ as an elementary divisor with positive multiplicity, and every elementary divisor of the form $(t-1)^{2m+1}$, $m \geq 2$, has even multiplicity, then that element is strongly real in $\U(n, \FF_q)$.
\end{proposition}

We now find some elements in $\U(n, \FF_q)$ which are not strongly real.

\begin{proposition} \label{32} A unipotent element of type $(3,2)$ in $\U(5, \FF_q)$, $q$ even, is not strongly real.
\end{proposition}
\begin{proof}  Define $\U(5, \FF_q)$ by the Hermitian form defined by $J = \left[ \begin{array}{cc} N_3 &  \\  & N_2 \end{array} \right]$.  Then one unipotent element of type $(3,2)$ in $\U(5, \FF_q)$ is of the form $g = \left[ \begin{array}{ccccc} 1 & a & b &   &  \\  & 1 & \bar{a} &   &  \\  &   &  1 &   & \\  &  &  & 1 & 1 \\  &  &  &  & 1 \end{array} \right]$, where $b + \bar{b} = a \bar{a}$ and $a \neq 0$.  Now assume that $s \in \U(5, \FF_q)$ is such that $s^2 = 1$ and $sgs = g^{-1}$.  Considering first that $sgs = g^{-1}$, and then that $s^2 = 1$, we obtain that $s$ is of the form $s = \left[ \begin{array}{ccccc} 1 & s_{12} & s_{13} & as_{25} & s_{15} \\  & 1 & s_{23}  &  & s_{25} \\  &  & 1 &  &  \\  &  s_{42} &  s_{43} & 1 & s_{45} \\  &  & \bar{a}s_{42}  &   & 1 \end{array} \right]$, where $\bar{a} s_{12} + a s_{23} = b + \bar{b}$.  Since ${^\top \bar{s}} J s = J$ and $s^2 = 1$, we have ${^\top \bar{s}}J = Js$.  Applying this fact, we find that $s$ must be of the form $s = \left[ \begin{array}{ccccc} 1 & s_{12} & s_{13} & as_{25} & s_{15} \\  & 1 & \overline{s_{12}} &  & s_{25} \\  &  & 1 &  & \\  &  \overline{s_{25}} & \overline{s_{15}} & 1 & s_{45} \\  &  & \bar{a} \overline{s_{25}} &  & 1 \end{array} \right]$, where $\overline{s_{13}} = s_{13}$ and $\overline{s_{45}} = s_{45}$.  Applying the fact that $s^2 = 1$ again, we find that $s_{12} = 0$.  But then $\bar{a} s_{12} + a s_{23} = b + \bar{b} = a \bar{a}$, so $\bar{a} s_{12} + a \overline{s_{12}} = 0 = a \bar{a}$.  This implies $a=0$, which is a contradiction.
\end{proof}

\begin{proposition} \label{3^r} A unipotent element of type $(3^r)$, where $r$ is odd, in $\U(n, \FF_q)$, where $n=3r$ and $q$ is even, is not strongly real.
\end{proposition}
\begin{proof} Define $\U(n, \FF_q)$ by the Hermitian form corresponding to $J = N_{3r}$.  Then one unipotent element in $\U(n, \FF_q)$ of type $(3^r)$ has the form $g = \left[ \begin{array}{ccc} I_r & aI_r & bI_r \\   &  I_r & \bar{a}I_r \\  &  & I_r \end{array} \right]$, where $b+ \bar{b} = a \bar{a}$ and $a \neq 0$.  Assume that $s \in \U(n, \FF_q)$ such that $s^2 = 1$ and $sgs = g^{-1}$.  Using the fact that $sgs = g^{-1}$ yields that $s$ must be of the form $s = \left[ \begin{array} {ccc} s_{11} & s_{12} & s_{13} \\  & s_{11} & s_{23} \\  &  & s_{11} \end{array} \right]$, where $(b + \bar{b}) s_{11}  = a \bar{a} s_{11}= \bar{a} s_{12} + a s_{23}$, and each $s_{ij}$ is $r$-by-$r$.  Applying that $s^2 = 1$ gives $s_{11}^2 = 1$, $s_{11} s_{12} = s_{12} s_{11}$, $s_{11} s_{23} = s_{23} s_{11}$, and $s_{11} s_{13} + s_{12} s_{23} + s_{13} s_{11} = 0$.  Applying the trace to the last equation, and using the fact that $\mathrm{tr}(s_{11} s_{13}) = \mathrm{tr}(s_{13} s_{11})$, we have $\tr(s_{12} s_{23}) = \tr(s_{23} s_{12})= 0$.

Now, since $s_{11}^2 = 1$, then $s_{11}$ is unipotent, and $\tr(s_{11}) = r$.  Also, since $s_{11}$ and $s_{12}$ commute and $s_{11}$ is unipotent, then $\tr(s_{11} s_{12}) = \tr(s_{12})$, which may be observed by putting $s_{11}$ and $s_{12}$ simultaneously in upper triangular form over an algebraic closure.  Since we are in characteristic two, then it follows by considering $s_{12}$ in upper triangular form over an algebraic closure that we also have $\tr(s_{12}^2) = \tr(s_{12})^2$.

From the fact that ${^\top \bar{s}} J s = J$ and $s^2 = 1$, we have ${^\top \bar{s}} J = J s$.  From this, we obtain $s_{23} = N_r {^\top \overline{s_{12}}} N_r$ (among other relations we will not need), from which it follows that $\tr(s_{23}) = \overline{ \tr(s_{12})}$.  By taking the trace of $a \bar{a} s_{11} = \bar{a} s_{12} + a s_{23}$, we have $a \bar{a} r = \bar{a} \tr(s_{12}) + a \overline{\tr(s_{12})}$.  We must thus have $\tr(s_{12}) \neq 0$, otherwise we have $a \bar{a} r = 0$, so $a = 0$.

From $a \bar{a} s_{11} = \bar{a} s_{12} + a s_{23}$, we have $a \bar{a} s_{11} s_{12} = \bar{a} s_{12}^2 + a s_{23} s_{12}$.  Taking the trace of both sides, and applying that $\tr(s_{11} s_{12}) = \tr(s_{12})$, $\tr(s_{12}^2) = \tr(s_{12})^2$, and $\tr(s_{23} s_{12}) = 0$, we have $a \bar{a} \tr(s_{12}) = \bar{a} \tr(s_{12})^2$.  Since $\tr(s_{12}) \neq 0$ and $a \neq 0$, we get $\tr(s_{12}) = a$.  Finally, since $a \bar{a} r = \bar{a} \tr(s_{12}) + a \overline{\tr(s_{12})}$, we have $a \bar{a} r = \bar{a} a + a \bar{a} = 0$, which implies $a = 0$, a contradiction.
\end{proof}

We may now apply Proposition \ref{IndStep} to Propositions \ref{32} and \ref{3^r} to obtain larger sets of classes which are not strongly real.  For example, from Proposition \ref{32}, we may show that any unipotent class of type $(k,l)$, where $k$ is odd and $l$ is even, and $k \geq l$, is not strongly real.  From Proposition \ref{3^r}, we find that any unipotent class of type $(k^r)$ where $k$ and $r$ are odd, is not strongly real, which generalizes \cite[Proposition 5.1]{GoVi08} that any regular unipotent element in $\U(n, \FF_q)$, where $n$ is odd and $q$ is even, is not strongly real.  

The most general statement we can obtain by applying Propositions \ref{UniRed} and \ref{IndStep} to Propositions \ref{32} and \ref{3^r} is as follows.  As the proof uses exactly the same type of induction argument used in the proof of Theorem \ref{MainThm}, we omit it.

\begin{proposition} \label{notstrong2}
Let $g \in \U(n, \FF_q)$, with $q$ even, and let $\mu$ be the partition corresponding to the elementary divisors of $g$ of the form $(t-1)^{\mu_i}$.  Suppose that one of the following holds:
\begin{enumerate}
\item The total number of odd parts of $\mu$ is odd, and if $k$ is the smallest odd part of $\mu$, and $l$ is the largest even part of $\mu$ (where $l=0$ if $\mu$ has no even parts), then $k-l \geq 3$.
\item The partition $\mu$ has exactly one odd part, say $k \geq 3$, and if $l$ is the largest even part of $\mu$, then $k - l = 1$ and $l$ has multiplicity one in $\mu$.
\end{enumerate}
Then the element $g$ is not strongly real in $\U(n, \FF_q)$.
\end{proposition}

In the case that $q$ is even, there of course must be either more strongly real classes in $\U(n, \FF_q)$ than are given in Proposition \ref{Real2}, or more classes which are not strongly real than are given in Proposition \ref{notstrong2}, and almost certainly more of both.  So the results in this section give only part of the picture.  We leave open the problem of a complete classification of strongly real classes in $\U(n,\FF_q)$ with $q$ even.

\bigskip

\noindent Zachary Gates, Department of Mathematics, University of Virginia, P. O. Box 400137, Charlottesville, VA  22904, USA, email: {\tt zg8bf@virginia.edu} \\
\\
\noindent Anupam Singh, IISER, Central Tower, Sai Trinity Building, near Garware Circle, Pashan, Pune 411021, India, email: {\tt anupamk18@gmail.com}\\
\\
\noindent C. Ryan Vinroot, Mathematics Department, College of William and Mary, P. O. Box 8795, Williamsburg, VA  23187-8795, USA, email: {\tt vinroot@math.wm.edu}

\end{document}